\documentclass[twoside,leqno]{article}
\usepackage{amssymb}
\usepackage{amsmath}
\usepackage{amsthm}
\usepackage{color}
\usepackage{amscd}
\definecolor{darkgreen}{rgb}{0,0.5,0}

\def\qed{\hfill$\square$}
\numberwithin{equation}{section}
\newtheorem{thm}[equation]{\sc Theorem}
\newtheorem{lem}[equation]{\sc Lemma}
\newtheorem{cor}[equation]{\sc Corollary}
\newtheorem{prop}[equation]{\sc Proposition}

\newtheoremstyle{notation}{3pt}{3pt}{}{}{\itshape}{:}{.5em}{\thmname{#1}}
\theoremstyle{notation}
\newtheorem{notation}{\it Notation}
\newtheorem{rem}{\it Remark}

\newtheorem{defin}{\it Definition}
\newtheorem{ex}{\it Example}
\newtheorem{algorithm}{\it Algorithm}

\def\epv {{$\mbox{}$\hfill ${\Box}$\vspace*{1.5ex} }}

\def\type{\mbox{\rm type\,}}

\def\mod{\mbox{{\rm mod}}}
\def\Hom{\mbox{\rm Hom}}

\def\Aut{\mbox{\rm Aut}}
\def\rad{\mbox{\rm rad}}
\def\soc{\mbox{\rm soc\,}}

\def\Im{\mbox{\rm Im}}

\def\rank{\mbox{\rm rank}}

\def\longarr#1#2{{\buildrel{#1} \over {\hbox to #2pt{\rightarrowfill}}}}

\def\D{\displaystyle}

\input prepictex   \input pictex    
\input postpictex

\newcounter{boxsize}
\newcounter{tempcounter}
\newcommand{\smallentryformat}{\scriptstyle\sf}
\newcommand\num[1]{\makebox(0,0){$\smallentryformat#1$}}

\newcommand\smbox{\put(0,0){\line(1,0){\value{boxsize}}}%
  \put(\value{boxsize},0){\line(0,1){\value{boxsize}}}%
  \put(0,0){\line(0,1){\value{boxsize}}}%
  \put(0,\value{boxsize}){\line(1,0){\value{boxsize}}}}
\newcommand\numbox[1]{\put(0,0)\smbox%
  \put(0,0){\makebox(\value{boxsize},\value{boxsize})[c]{%
      $\smallentryformat#1$}}}
\newcommand\singllebox[1]{\raisebox{-.4ex}{\begin{picture}(4,0)\setcounter{boxsize}{3}%
    \put(0,0)\smbox%
    \put(0,0){\makebox(\value{boxsize},\value{boxsize})[c]{%
      $\scriptstyle\sf#1$}}\end{picture}}}

\def\sbullet{\makebox(0,0){$\scriptstyle\bullet$}}

\newcommand\boxes[2]{\ifthenelse{#2=3}{$\scriptstyle P_2^{#1}$}{%
                                       $\scriptstyle P_{#2}^{#1}$}}

\begin{document}
\thispagestyle{empty}


\bigskip\bigskip
\begin{center}
{\large\bf The boundary of the irreducible components %
  for invariant subspace varieties}

\medskip
Justyna Kosakowska and Markus Schmidmeier\footnote{The first named author is partially supported 
    by Research Grant No.\ DEC-2011/02/A/ ST1/00216
    of the Polish National Science Center. 
This research is partially supported 
    by a~Travel and Collaboration Grant from the Simons Foundation (Grant number
    245848 to the second named author).}

\bigskip \parbox{10cm}{\footnotesize{\bf Abstract:} Given partitions $\alpha$, $\beta$, $\gamma$, the short 
exact sequences 
$$0\longrightarrow N_\alpha \longrightarrow N_\beta 
   \longrightarrow N_\gamma \longrightarrow 0$$
of nilpotent linear operators of Jordan types $\alpha$, $\beta$, $\gamma$,
respectively, define a~constructible subset 
$\mathbb V_{\alpha,\gamma}^\beta$ of an affine variety.

Geometrically, the varieties $\mathbb V_{\alpha,\gamma}^\beta$ are of particular interest
as they occur naturally and since they typically consist of several irreducible components.
In fact, each  Littlewood-Richardson tableau $\Gamma$ of shape $(\alpha,\beta,\gamma)$
contributes one irreducible component $\overline{\mathbb V}_\Gamma$.

We consider the partial order 
$\Gamma\leq_{\sf boundary}\widetilde{\Gamma}$ 
on LR-tableaux which is 
the transitive closure of the relation given by
$\mathbb V_{\widetilde{\Gamma}}\cap \overline{\mathbb V}_\Gamma\neq \emptyset$.
In this paper we compare the boundary relation with partial orders
given by algebraic, combinatorial and geometric conditions.
It is known that 
in the case where the parts of $\alpha$ are at most two,
all those partial orders are equivalent.  
 We prove that those partial orders are also equivalent
 in the case where $\beta\setminus\gamma$ is a~horizontal and
 vertical strip.
Moreover, we discuss how the orders differ in general.
}

\medskip \parbox{10cm}{\footnotesize{\bf MSC 2010:
14L30, 16G20, 47A15.
}}

\medskip \parbox{10cm}{\footnotesize{\bf Key words:
Nilpotent operator, Invariant subspace, Partial order, Degeneration, Littlewood-Richardson tableau.
}}
\end{center}

\section{Introduction}

Often in geometry, naturally occurring conditions define subsets of varieties
which are either very big in size or tiny.
For example, among all linear operators acting on a~given finite dimensional 
vector space, the invertible ones form an open and dense subset.
And so do, among all nilpotent operators, those which have only one 
Jordan block. 
A notable exception to this rule occurs 
in the variety of short exact sequences of
nilpotent linear operators; it is partitioned, by means of Littlewood-Richardson
tableaux, into components of equal dimension.  They are the topic of this paper.

Throughout we assume that $k$ is an algebraically closed field.
Each nilpotent $k$-linear 
operator is given uniquely, up to isomorphy, as a~$k[T]$-module
$N_\alpha=\bigoplus_{i=1}^s k[T]/(T^{\alpha_i})$ for some 
partition $\alpha=(\alpha_1,\ldots,\alpha_s)$ which represents the sizes of its Jordan blocks.

The Theorem of Green and Klein \cite{gk} states that for given partitions
$\alpha$, $\beta$, $\gamma$, there exists a~short exact sequence 
$$0\longrightarrow N_\alpha\longrightarrow N_\beta
   \longrightarrow N_\gamma\longrightarrow 0$$
of nilpotent linear operators if and only if there is a~Littlewood-Richardson
(LR-) tableau of shape $(\alpha,\beta,\gamma)$. 
The collection of all such short exact sequences forms 
a~variety $\mathbb V_{\alpha,\gamma}^\beta(k)$ which 
can be partitioned using LR-tableaux, as follows.
Consider the affine variety $\Hom_k(N_\alpha,N_\beta)$ of $k$-linear maps endowed with the Zariski topology,
and assume that all subsets carry the induced topology. Define
\begin{eqnarray*} \mathbb V_{\alpha,\gamma}^\beta(k)\;=\;\big\{f:N_\alpha\to N_\beta & \; :\; &
  f \;\text{monomorphism of}\; k[T]\text{-modules}\\
  & & \qquad\text{with cokernel isomorphic to}\; N_\gamma\big\}.
\end{eqnarray*}
The irreducible 
components of  $\mathbb V_{\alpha,\gamma}^\beta(k)$ 
are counted by the Littlewood-Richard\-son coefficient.
Namely, to each monomorphism in $\mathbb V_{\alpha,\gamma}^\beta$ one can associate
an LR-tableau $\Gamma$ of shape $(\alpha,\beta,\gamma)$, 
as we will see in Section~\ref{section-LR}.
The subset $\mathbb V_\Gamma$ of $\Hom_k(N_\alpha,N_\beta)$ 
of all such monomorphisms is constructible and irreducible.
All the $\mathbb V_\Gamma$ have the same dimension.  
We denote by $\overline{\mathbb V}_\Gamma$
the closure of $\mathbb V_\Gamma$ in $\mathbb V_{\alpha,\gamma}^\beta$; 
the sets $\overline{\mathbb V}_\Gamma$
define the irreducible components of $\mathbb V_{\alpha,\gamma}^\beta$, they are indexed
by the set $\mathcal T_{\alpha,\gamma}^\beta$ of all LR-tableaux of shape $(\alpha,\beta,\gamma)$
(see \cite[Theorem 4.3]{leeuwen} and \cite{maeda}). 

Our aim in this paper is to shed light on the geometry in the variety
$$\mathbb V_{\alpha,\gamma}^\beta\;=\;
     \bigcup^\bullet_{\Gamma\in\mathcal T_{\alpha,\gamma}^\beta}\mathbb V_\Gamma;$$
by studying the boundary relation given as follows.
\begin{equation}
\Gamma\preceq_{\sf boundary}\widetilde{\Gamma} \quad\Leftrightarrow\quad \mathbb V_{\widetilde{\Gamma}}
                 \cap\overline{\mathbb V}_\Gamma\neq\emptyset
                 \qquad\text{where}\qquad \Gamma,\widetilde{\Gamma}\in\mathcal T_{\alpha,\gamma}^\beta. 
\label{eq-closure}\end{equation} 
 Obviously, $\preceq_{\sf boundary}$ is reflexive and we will see that it is anti-symmetric.
 We denote by $\leq_{\sf boundary}$ the transitive closure of $\preceq_{\sf boundary}$.
 In general, for a~reflexive and anti-symmetric relation $\preceq_{\sf x}$ we denote its transitive closure
   by $\leq_{\sf x}$.

On the set 
$$\mathcal P=\mathcal{P}_{\alpha,\gamma}^\beta=\{\mathbb{V}_\Gamma\; :\; \Gamma\in \mathcal{T}_{\alpha,\gamma}^\beta\}$$
of 
irreducible components of the representations space $\mathbb V_{\alpha,\gamma}^\beta$, 
there are several relations of algebraic, geometric and combinatorial nature:  the hom- and the ext-order,
the degeneration order and the boundary condition, the box relation and the dominance order.
By taking the reflexive and transitive closure of each relation, if necessary, we obtain six partial orders on the set $\mathcal P$.  
It is the aim of the paper to compare those partial orders.

Given two partial orders  
$(\mathcal P, \leq_x), (\mathcal P, \leq_y)$ on the same set, 
we write $\leq_x\;\to\; \leq_y$ and  say that $(\mathcal P, \leq_x)$ is {\it contained} in $(\mathcal P, \leq_y)$ 
if $P \leq_x Q$  implies  $P \leq_y Q$ for all $P,Q \in\mathcal P$.  
With respect to the containment relation, we obtain the following diagram (whenever the box-relation is defined):

$$
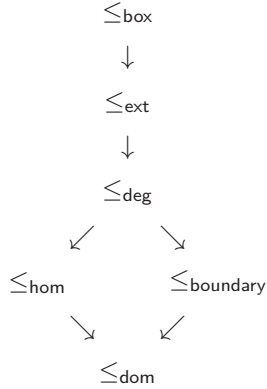
\setlength\unitlength{.6mm}
\begin{picture}(40,80)
\put(20,0){\makebox(0,0){$\leq_{\sf dom}$}}
\put(0,20){\makebox(0,0){$\leq_{\sf hom}$}}
\put(40,20){\makebox(0,0){$\leq_{\sf boundary}$}}
\put(20,40){\makebox(0,0){$\leq_{\sf deg}$}}
\put(20,60){\makebox(0,0){$\leq_{\sf ext}$}}
\put(20,80){\makebox(0,0){$\leq_{\sf box}$}}
\put(20,50){\makebox(0,0){$\downarrow$}}
\put(20,70){\makebox(0,0){$\downarrow$}}
\put(10,30){\makebox(0,0){$\swarrow$}}
\put(30,30){\makebox(0,0){$\searrow$}}
\put(10,10){\makebox(0,0){$\searrow$}}
\put(30,10){\makebox(0,0){$\swarrow$}}
\end{picture}
$$

We give examples that, in general, no two relations are equal.  
However, if $\beta\setminus\gamma$ is a horizontal and vertical strip, then they are all equal.

Since the set $\mathcal{P}$ of irreducible components of $\mathbb{V}_{\alpha,\gamma}^\beta$ is in bijection
with the set $\mathcal{T}_{\alpha,\gamma}^\beta$, we will work with posets $(\mathcal{T}_{\alpha,\gamma}^\beta,\leq_{\sf x})$
instead of $(\mathcal{P},\leq_{\sf x})$.

We are ready to present  the main results of the paper.

\subsection{Two algebraic tests}

The algebraic group $G=\Aut_{k[T]}(N_\alpha)\times\Aut_{k[T]}(N_\beta)$  
acts on 
$\mathbb V_{\alpha,\gamma}^\beta$ via $(a,b)\cdot f=bfa^{-1}$.  The orbits of this group action
are in one-to-one correspondence with the isomorphism classes of embeddings
$f:N_\alpha\to N_\beta$.

We consider the following reflexive relation for LR-tableaux.
We say $\Gamma\preceq_{\sf deg}\widetilde{\Gamma}$ if there are embeddings $f\in\mathbb V_\Gamma$,
$\tilde f\in\mathbb V_{\widetilde{\Gamma}}$ such that $f\leq_{\sf deg}\tilde f$, that is,
$\mathcal O_{\tilde f}\subset\overline{\mathcal O}_f$, where $\mathcal{O}_f$
is the orbit of $f$ under the action of $G$ on $\mathbb V_{\alpha,\gamma}^\beta$.

The degeneration relation is under control algebraically as the ext-relation $\leq_{\sf ext}$ implies
the deg-relation $\leq_{\sf deg}$, which in turn implies the hom-relation $\leq_{\sf hom}$ (see Section \ref{section-deg}). 
We show that the boundary relation implies the restriction $\leq_{\sf hom-picket}$
of the hom order to certain objects called {\it pickets.}

In the diagram below, the relations introduced so far on the set
$\mathcal T_{\alpha,\gamma}^\beta$ are ordered vertically by containment.

$$
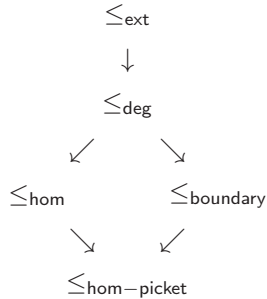
\setlength\unitlength{.6mm}
\begin{picture}(40,60)
\put(20,0){\makebox(0,0){$\leq_{\sf hom-picket}$}}
\put(0,20){\makebox(0,0){$\leq_{\sf hom}$}}
\put(40,20){\makebox(0,0){$\leq_{\sf boundary}$}}
\put(20,40){\makebox(0,0){$\leq_{\sf deg}$}}
\put(20,60){\makebox(0,0){$\leq_{\sf ext}$}}
\put(20,50){\makebox(0,0){$\downarrow$}}
\put(10,30){\makebox(0,0){$\swarrow$}}
\put(30,30){\makebox(0,0){$\searrow$}}
\put(10,10){\makebox(0,0){$\searrow$}}
\put(30,10){\makebox(0,0){$\swarrow$}}
\end{picture}
$$

We show that the restriction of the hom-order to pickets is an anti-symmetric relation.
As a consequence, all the relations in the diagram are partial orders
on $\mathcal T_{\alpha,\gamma}^\beta$.
We have algebraic tests both for the validity and for the failure
of the boundary relation: 

\begin {thm}\label{theorem-ext-bound}
Suppose $\alpha$, $\beta$, $\gamma$ are partitions.
The following implications hold for 
LR-tableaux $\Gamma$, $\widetilde\Gamma$ of shape $(\alpha,\beta,\gamma)$:

$$\Gamma\leq_{\sf ext}\widetilde{\Gamma} \quad\Longrightarrow\quad
   \Gamma\leq_{\sf boundary}\widetilde{\Gamma} \quad\Longrightarrow\quad
   \Gamma\leq_{\sf hom-picket}\widetilde{\Gamma}.$$
\end {thm}

We present proofs in Section~\ref{section-deg}.

\subsection{Two combinatorial criteria}
%
%
On the set $\mathcal T_{\alpha,\gamma}^\beta$, there are two partial orders of
combinatorial nature.  The {\it dominance relation} $\leq_{\sf dom}$
is given by the natural partial orders of the partitions defining the tableaux.
The second relation is the {\it box-order} $\leq_{\sf box}$, it is given by repeatedly exchanging two 
entries in the tableau in such a~way that the smaller entry moves up and such
that the lattice permutation condition is preserved.
We introduce the two orders formally in Section~\ref{section-lr-orders}.

\bigskip
The following result presents a~necessary and a~sufficient criterion of
combinatorial nature
for two LR-tableaux to be in boundary relation:

\begin {thm}\label{theorem-bound-dom}
Given partitions $\alpha$, $\beta$, $\gamma$,
the following implications hold for 
LR-tableaux $\Gamma$, $\widetilde\Gamma$ of shape $(\alpha,\beta,\gamma)$.
\begin{enumerate}
\item[(a)] If $\Gamma\leq_{\sf boundary}\widetilde\Gamma$ 
  then $\Gamma\leq_{\sf dom}\widetilde\Gamma$.
\item[(b)] Suppose $\beta\setminus\gamma$ is a~horizontal strip.
  If $\Gamma\leq_{\sf box}\widetilde\Gamma$ then $\Gamma\leq_{\sf boundary}\widetilde\Gamma$.
\end{enumerate}
\end {thm}

We show in Section \ref{section-closure2box} that the dominance relation is in fact
equivalent to the restriction of the hom-order to pickets.
The second part follows from a result in \cite{ks-ext}:

\begin {prop}\label{proposition-box-ext}
Suppose $\Gamma$, $\widetilde\Gamma$ are LR-tableaux which have the same shape and which
are horizontal strips.  If $\Gamma\leq_{\sf box}\widetilde\Gamma$ then 
$\Gamma\leq_{\sf ext}\widetilde\Gamma$.
\end {prop}

\subsection{Horizontal and vertical strips}

Of particular interest is the case where the partitions are such 
that $\beta\setminus\gamma$ is both a~horizontal and a vertical strip.
In this situation, the combinatorial relations $\leq_{\sf box}$ and $\leq_{\sf dom}$
are in fact equivalent.
In \cite{kst} we give two proofs for this statement; below
in Section~\ref{section-lr-orders} we sketch the algorithmic approach 
in one of them.  We deduce the following result.

\begin {thm}\label{theorem-horizontal-vertical}
Suppose $\alpha$, $\beta$, $\gamma$ are partitions such that 
$\beta\setminus\gamma$ is a~horizontal and vertical strip. 
The following relations are partial orders which are equivalent to each other.
$$\leq_{\sf box},\;\leq_{\sf ext},\;\leq_{\sf deg},\;\leq_{\sf hom},\;
   \leq_{\sf boundary}, \; \leq_{\sf dom}.$$
\end {thm}

For comparison we note that there is a related result 
about the six partial orders
in the case where all parts of $\alpha$ are at most two.
In this situation
the orbits and the boundary relation are given combinatorially in terms of
arc diagrams and of resolution of crossings, respectively \cite{ks-hall,ks-survey}:

\begin {thm}
Suppose $\alpha$, $\beta$, $\gamma$ are partitions such that all parts of $\alpha$ are
at most two. 
The relations $\leq_{\sf dom}$, $\leq_{\sf hom}$, $\leq_{\sf boundary}$,
  $\leq_{\sf deg}$, $\leq_{\sf ext}$, $\leq_{\sf box}$ are all partial orders
  which are equivalent to each other.
\end {thm}

\subsection{Related results}

The Theorem of Gerstenhaber and Hesselink shows 
that the natural partial order of partitions is equivalent
to the degeneration order of nilpotent linear operators, 
see \cite{gerst,hess,kraft}. We investigate a~similar problem: connections
of the dominance order of LR-tableaux with the boundary order defined below. 
Also extensions of nilpotent linear operators
are of interest as they are connected with the classical Hall algebras 
and Hall polynomials, see \cite{macd}. 
Well understood are generic extensions and their relationships 
with the specializations to $q=0$ of the Ringel-Hall algebras, 
see \cite{dengdu,dengdumah,kos12,reineke,reineke1}.

\subsection{Organization of this paper}

In Section \ref{section-LR}, we describe how partitions and tableaux describe 
short exact sequences of linear operators, or equivalently of {\it embeddings}
or {\it invariant subspaces} of linear operators.
Moreover, we introduce pickets as special types of embeddings.

\smallskip

In Section \ref{section-closure}, we show that the boundary relation in Formula~(\ref{eq-closure})
implies the dominance order  (Part (a) of Theorem~\ref{theorem-bound-dom}).
 As a consequence we obtain that the boundary relation is anti-symmetric. 
 We present an~example showing that $\preceq_{\sf boundary}$ may not be transitive.
  Moreover, 
 in (\ref{section-bound-dom-not-equiv}) we provide an~example showing that  
 the assumption that $\beta\setminus\gamma$ is a~vertical strips is
     necessary in Theorem \ref{theorem-horizontal-vertical}.

\smallskip
In Section \ref{section-deg}, we adapt the ext- deg- and hom-relations for modules to tableaux.
As for modules, the ext-order implies the degeneration order, which implies the hom-order.
Moreover, the hom-relation implies the dominance order.
This completes the proof of Theorem~\ref{theorem-ext-bound}.
Using results given in \cite{ks-ext} and in \cite{kst},
we show part (b) of Theorem~\ref{theorem-bound-dom}
and complete the proof of Theorem~\ref{theorem-horizontal-vertical}.

\section{Littlewood-Richardson tableaux}
\label{section-LR}

Given three partitions, $\alpha$, $\beta$, $\gamma$, we consider the set 
$\mathcal T_{\alpha,\gamma}^\beta$ of 
all Littlewood-Richardson tableaux of shape $(\alpha,\beta,\gamma)$.
We define the dominance order on the set $\mathcal T_{\alpha,\gamma}^\beta$.
Moreover, we introduce the LR-tableau of a~short exact sequence, and determine the tableaux
for certain types of short exact sequences, in particular pickets.
For the case where the skew diagram $\beta\setminus\gamma$ is a~horizontal strip, 
we also introduce the box-order.

\subsection{Combinatorial orders on the set of LR-tableaux}
\label{section-lr-orders}

\begin{notation}
Recall that a~{\it partition} $\alpha=(\alpha_1,\ldots,\alpha_s)$
is a~finite non-increasing sequence of natural numbers;
we picture $\alpha$ by its Young diagram which consists of $s$ columns of length
given by the parts of $\alpha$.
The {\it transpose} $\alpha'$ of $\alpha$ is given by the formula
$$\alpha'_j=\#\{i:\alpha_i\geq j\},$$
it is pictured by the transpose of the Young diagram for $\alpha$.
Two partitions $\alpha$, $\widetilde{\alpha}$ are in the {\it natural partial order,}
  in symbols $\alpha\leq_{\sf nat}\widetilde{\alpha}$, if the  inequality 
  $$\alpha'_1+\cdots+\alpha'_j\leq\widetilde{\alpha}'_1+\cdots+\widetilde{\alpha}'_j$$
  holds for each $j$.

\bigskip
Given three partitions $\alpha$, $\beta$, $\gamma$, 
an {\it LR-tableau} of shape $(\alpha,\beta,\gamma)$ is a~Young diagram of
shape $\beta$ in which the region $\beta\setminus\gamma$ contains
$\alpha_1'$ entries $\singllebox 1$, $\ldots$, $\alpha_t'$ entries $\singllebox t$, 
where $t=\alpha_1$ is the largest entry, such that
\begin{itemize}
\item in each row, the entries are weakly increasing,
\item in each column, the entries are strictly increasing,
\item for each $\ell>1$ and for each column $c$:
  on the right hand side of $c$, the number of entries $\ell-1$
  is at least the number of entries $\ell$.
\end{itemize}
The skew diagram  $\beta\setminus\gamma$ 
is said to be a~{\it horizontal strip} if $\beta_i\leq \gamma_i+1$ 
holds for all $i$, and a~{\it vertical strip} if $\beta'\setminus\gamma'$
is a~horizontal strip.
\end{notation}

\smallskip

\begin {ex}  Let $\alpha=(3,2)$, $\beta=(4,3,3,2,1)$, $\gamma=(3,2,2,1)$.
Then the transpose of $\alpha$ is $\alpha'=(2,2,1)$, so we have to fill the skew
diagram $\beta\setminus\gamma$ with two $\singllebox1$'s, two $\singllebox 2$'s, 
and one $\singllebox 3$.
Due to the conditions on an LR-tableau, this can be done in exactly two ways.

$$ \setcounter{boxsize}{3}
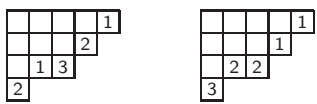
\begin{picture}(18,12)(0,2)
\multiput(0,12)(3,0)5{\smbox}
\put(12,12){\numbox{1}}
\multiput(0,9)(3,0)4{\smbox}
\put(9,9){\numbox{2}}
\multiput(0,6)(3,0)3{\smbox}
\put(3,6){\numbox{1}}
\put(6,6){\numbox{3}}
\put(0,3){\numbox{2}}
\end{picture}
\;\;\;\;\;\;\;\;
\setcounter{boxsize}{3}
\begin{picture}(18,12)(0,2)
\multiput(0,12)(3,0)5{\smbox}
\put(12,12){\numbox{1}}
\multiput(0,9)(3,0)4{\smbox}
\put(9,9){\numbox{1}}
\multiput(0,6)(3,0)3{\smbox}
\put(3,6){\numbox{2}}
\put(6,6){\numbox{2}}
\put(0,3){\numbox{3}}
\end{picture}
$$
In this example, $\beta\setminus\gamma$ is~a~horizontal but not a~vertical strip.
\end {ex}

\begin{notation}
One can represent an LR-tableau $\Gamma$ by a~sequence of partitions
$$\Gamma=[\gamma^{(0)},\ldots,\gamma^{(t)}]$$
where $\gamma^{(i)}$ denotes the region in the Young diagram $\beta$ which contains the entries
$\singllebox{}$, $\singllebox{1}$, $\ldots$, $\singllebox{i}$.
If $\Gamma$ has shape $(\alpha,\beta,\gamma)$, then $\gamma=\gamma^{(0)}$,
$\beta=\gamma^{(t)}$, and $\alpha_i'=|\gamma^{(i)}\setminus\gamma^{(i-1)}|$ for $i=1,\ldots,t$.
\end{notation}

\smallskip
In the example above, the first tableau is given by the sequence of partitions
$\Gamma=[(3,2,2,1),(3,3,2,1,1),(4,3,2,2,1),(4,3,3,2,1)]$.

We introduce two partial orders on the set 
$\mathcal T_{\alpha,\gamma}^\beta$ of all LR-tableaux of 
shape $(\alpha,\beta,\gamma)$.

\begin {defin}
Two LR-tableaux $\Gamma=[\gamma^{(0)},\ldots,\gamma^{(t)}]$, 
  $\widetilde{\Gamma}=[\widetilde{\gamma}^{(0)},\ldots,\widetilde{\gamma}^{(t)}]$ 
  of the same shape
  are in the {\it dominance order,} in symbols $\Gamma\leq_{\sf dom}\widetilde{\Gamma}$,
  if for each $i$, the corresponding partitions 
  $\gamma^{(i)}$, $\widetilde{\gamma}^{(i)}$ are in the natural partial order, i.e. 
   $\gamma^{(i)}\leq_{\sf nat}\widetilde{\gamma}^{(i)}$.
\end {defin}

\begin {defin}
Suppose $\Gamma$, $\widetilde\Gamma$ are LR-tableaux of the same shape which we assume
to be a~horizontal strip.
We say $\widetilde\Gamma$ is obtained from $\Gamma$ by {\it a~box move} if
after two entries in $\Gamma$ have been exchanged in such
a way that the smaller entry is in the higher position
                                                        in $\widetilde\Gamma$, 
                                                   we obtain $\widetilde\Gamma$ by re-sorting the 
  list of
columns if necessary.
We denote by $\leq_{\sf box}$ the partial order generated by box moves.
\end {defin}

Here is an example:
$$\raisebox{4mm}{$\Gamma:$} \quad
\begin{picture}(15,12)(0,3)
\multiput(0,12)(3,0)4{\smbox}
\put(12,12){\numbox{1}}
\multiput(0,9)(3,0)3{\smbox}
\put(9,9){\numbox{2}}
\multiput(0,6)(3,0)2{\smbox}
\put(6,6){\numbox{3}}
\put(0,3){\smbox}
\put(3,3){\numbox{1}}
\put(0,0){\numbox{2}}
\end{picture}
\quad\raisebox{3mm}{$<_{\sf box}$} \;\;\;
\raisebox{4mm}{$\widetilde\Gamma:$} \quad
\begin{picture}(15,12)(0,3)
\multiput(0,12)(3,0)4{\smbox}
\put(12,12){\numbox{1}}
\multiput(0,9)(3,0)3{\smbox}
\put(9,9){\numbox{1}}
\multiput(0,6)(3,0)2{\smbox}
\put(6,6){\numbox{2}}
\put(0,3){\smbox}
\put(3,3){\numbox{2}}
\put(0,0){\numbox{3}}
\end{picture}
$$

\begin {rem}
 In \cite{ks-ext} the box-order is defined in a~more general case:  in the case when
 LR-tableaux are unions of so called columns. For simplicity, we present definitions and results for
 horizontal strips.
\end {rem}

\begin {lem}
For LR-tableaux of the same shape, the $\leq_{\sf box}$-order 
implies the $\leq_{\sf dom}$- order.
\end {lem}

  \noindent {\bf Proof.}
Suppose the LR-tableau 
$\widetilde\Gamma=[\widetilde\gamma^{(0)},\ldots,\widetilde\gamma^{(t)}]$ 
is obtained from $\Gamma=[\gamma^{(0)},\ldots,\gamma^{(t)}]$ by a~box move
based on entries $i$ and $j$ with, say, $i<j$.
The process of reordering the entries in each row
will not affect entries less than $i$ or larger than $j$, so
the partitions $\gamma^{(0)},\ldots,\gamma^{(i-1)}$,
and $\gamma^{(j)},\ldots,\gamma^{(t)}$ remain unchanged.
The partitions $\gamma^{(\ell)}$, $\widetilde\gamma^{(\ell)}$ for $i\leq \ell<j$ 
are different and satisfy $\gamma^{(\ell)}<_{\sf nat}\widetilde\gamma^{(\ell)}$ 
(since the defining partial sums can only increase).  
This shows that $\Gamma<_{\sf dom}\widetilde\Gamma$.
  \epv

The converse does not always hold, not even for horizontal strips:

\begin {ex}

Let $\beta=(4,3,3,2,1)$, $\gamma=(3,2,2,1)$ and $\alpha=(3,2)$. 
We have seen that there are two
LR-tableaux of type $(\alpha,\beta,\gamma)$.
 They are incomparable in $\leq_{\sf box}$-relation, but
$$ \setcounter{boxsize}{3}
\begin{picture}(18,12)(0,3)
\multiput(0,12)(3,0)5{\smbox}
\put(12,12){\numbox{1}}
\multiput(0,9)(3,0)4{\smbox}
\put(9,9){\numbox{2}}
\multiput(0,6)(3,0)3{\smbox}
\put(3,6){\numbox{1}}
\put(6,6){\numbox{3}}
\put(0,3){\numbox{2}}
\end{picture}
\raisebox{4mm}{$\leq_{\sf dom}$}\;\;
\setcounter{boxsize}{3}
\begin{picture}(18,12)(0,3)
\multiput(0,12)(3,0)5{\smbox}
\put(12,12){\numbox{1}}
\multiput(0,9)(3,0)4{\smbox}
\put(9,9){\numbox{1}}
\multiput(0,6)(3,0)3{\smbox}
\put(3,6){\numbox{2}}
\put(6,6){\numbox{2}}
\put(0,3){\numbox{3}}
\end{picture}
$$
 \end {ex}

However for horizontal and vertical strips, the two partial orders are equivalent
\cite{kst}:

\begin {thm}\label{theorem-dom-box}
Suppose $\alpha$, $\beta$, $\gamma$ are partitions such that $\beta\setminus\gamma$
is a~horizontal and vertical strip.
Then the two partial orders $\leq_{\sf dom}$, $\leq_{\sf box}$ are equivalent
on $\mathcal T_{\alpha,\gamma}^\beta$. \qed
\end {thm}

In \cite{kst} we present two proofs of the fact that $\leq_{\sf dom}$ implies $\leq_{\sf box}$
(for horizontal and vertical strips). Both are algorithmic. Below we present one of these algorithms 
without any proof of its correctness.
The reader is referred to \cite{kst} for details and proofs.

 \begin{algorithm} For an LR-tableau $\Gamma$ we denote by $\omega(\Gamma)$ the
list of entries when read from left to right.
Clearly, $\Gamma$ is determined uniquely by 
its shape and by the list of its entries.

{\bf Input:} Two LR-tableaux $\Gamma$, $\widetilde\Gamma$ of shape $(\alpha,\beta,\gamma)$ such that
$\beta\setminus\gamma$ is a~horizontal and vertical strip and
such that $\Gamma<_{\sf dom}\widetilde\Gamma$. 

{\bf Output:} An LR-tableau $\widehat\Gamma$ of the shape $(\alpha,\beta,\gamma)$ such that
$\Gamma\leq_{\sf dom}\widehat{\Gamma}$ and $\widehat{\Gamma}<_{\sf box}\widetilde{\Gamma}$.

{\bf Step 1.} Find the smallest $k$ such that $\omega(\Gamma)_k\neq\omega(\widetilde\Gamma)_k$ and put
   $x=\omega(\Gamma)_k$. 
 
{\bf Step 2.} Choose the minimal $m\geq k+1$ such that $x=\omega(\widetilde\Gamma)_m$. 

{\bf Step 3.} Let $y=\min\{\omega(\widetilde\Gamma)_i>x:k\leq i < m\}$.

{\bf Step 4.} Choose $k\leq l < m$ such that $y=\omega(\widetilde\Gamma)_l$.

{\bf Step 5.} Define $\widehat\Gamma$ such that $\omega(\widehat\Gamma)_i=\omega(\widetilde\Gamma)_i$,
for $i\neq l,m$, and $\omega(\widehat\Gamma)_l=x$, $\omega(\widehat\Gamma)_m=y$. 

\end{algorithm}

\smallskip

\begin {ex} Let $\beta=(6,5,4,3,2,1)$, $\gamma=(5,4,3,2,1)$ and $\alpha=(3,2,1)$.
Consider two LR-tableaux $\Gamma$ and $\widetilde\Gamma$ of the shape $(\alpha,\beta,\gamma)$
such that $\omega(\Gamma)=(1,3,2,2,1,1)$ and $\omega(\widetilde\Gamma)=(2,3,2,1,1,1)$.
It is straightforward to check that $\beta\setminus\gamma$ is a~horizontal and vertical strip
and $\Gamma<_{\sf dom}\widetilde\Gamma$. 
$$\raisebox{5mm}{$\Gamma:$} \quad
\begin{picture}(18,15)(0,3)
\multiput(0,15)(3,0)5{\smbox}
\put(15,15){\numbox{1}}
\multiput(0,12)(3,0)4{\smbox}
\put(12,12){\numbox{1}}
\multiput(0,9)(3,0)3{\smbox}
\put(9,9){\numbox{2}}
\multiput(0,6)(3,0)2{\smbox}
\put(6,6){\numbox{2}}
\put(0,3){\smbox}
\put(3,3){\numbox{3}}
\put(0,0){\numbox{1}}
\end{picture}
\quad\raisebox{5mm}{$<_{\sf dom}$} \;\;\;
\raisebox{5mm}{$\widetilde\Gamma:$} \quad
\begin{picture}(18,15)(0,3)
\multiput(0,15)(3,0)5{\smbox}
\put(15,15){\numbox{1}}
\multiput(0,12)(3,0)4{\smbox}
\put(12,12){\numbox{1}}
\multiput(0,9)(3,0)3{\smbox}
\put(9,9){\numbox{1}}
\multiput(0,6)(3,0)2{\smbox}
\put(6,6){\numbox{2}}
\put(0,3){\smbox}
\put(3,3){\numbox{3}}
\put(0,0){\numbox{2}}
\end{picture}
$$
We apply the algorithm. Note that $k=1$, $x=1$ and $m=4$.
Now we can choose $y=\omega(\widetilde\Gamma)_1=2$ or $y=\omega(\widetilde\Gamma)_3=2$.
If we choose $y=\omega(\widetilde\Gamma)_1$, then $\widehat\Gamma=\Gamma$. In the second case,
i.e. if $y=\omega(\widetilde\Gamma)_3$, we get $\omega(\widehat\Gamma)=(2,3,1,2,1,1)$. It is easy to see that
$\Gamma<_{\sf dom}\widehat\Gamma$ and we can continue. 
\end {ex}

\subsection{The LR-tableau of a~short exact sequence}

\begin{notation}
By a~{\it nilpotent operator} we understand a~pair $(V,T)$ where
$V$ is a~finite dimensional $k$-vector space and $T:V\to V$ 
a~$k$-linear nilpotent operator. 
Each such pair is determined uniquely, up to isomorphy, by the partition
$\alpha=(\alpha_1,\ldots,\alpha_s)$ which records the sizes of the Jordan
blocks. We consider $(V,T)$ as the module over the 
polynomial ring $$N_\alpha:=\bigoplus_{i=1}^s k[T]/(T^{\alpha_i}).$$
Conversely, given a~$k[T]$-module $M$ on which the variable $T$ acts nilpotently, the 
transpose of the partition $\beta$ such that $M\cong N_\beta$
is given by $\beta'_\ell=\dim\frac{T^{\ell-1}M}{T^\ell M}$.

Given three partitions $\alpha$, $\beta$, $\gamma$, there is a~short exact
sequence $E:0\to N_\alpha\to N_\beta\to N_\gamma\to 0$ if and only if
there is an LR-tableau of shape $(\alpha,\beta,\gamma)$ \cite{gk}.
The {\it tableau} $\Gamma$ {\it corresponding to the sequence} $E$ is obtained as follows.
Let $B$ be the $k[T]$-module $N_\beta$ and $A$ the submodule given by the
image of the monomorphism $N_\alpha\to N_\beta$. The partitions defining
$\Gamma=[\gamma^{(0)},\ldots,\gamma^{(t)}]$, where $t=\alpha_1$, are obtained as the isomorphism
types of the nilpotent operators
\cite[II, (1.4)]{macd}: $$N_{\gamma^{(i)}}= B/T^iA.$$
\end{notation}

\begin {defin}
Given two partitions $\gamma$, $\widetilde\gamma$, the {\it union}
$\gamma\cup\widetilde\gamma$ has as Young diagram the sorted union of the columns
in the Young diagrams for $\gamma$ and $\widetilde\gamma$, in symbols,
$(\gamma\cup\widetilde\gamma)'_i= \gamma'_i+\widetilde\gamma'_i$.

\smallskip
For two tableaux $\Gamma=[\gamma^{(0)},\ldots,\gamma^{(s)}]$, 
$\widetilde\Gamma=[\widetilde\gamma^{(0)},\ldots,\widetilde\gamma^{(t)}]$,
the {\it union} of the tableaux is given 
rowwise:
$$\Gamma\cup\widetilde\Gamma\;=\;
  [\gamma^{(0)}\cup\widetilde\gamma^{(0)},\ldots,\gamma^{(m)}\cup\widetilde\gamma^{(m)}]$$
where $m=\max\{s,t\}$ and $\gamma^{(i)}=\gamma^{(s)}$ for $i>s$ and 
$\widetilde\gamma^{(i)}=\widetilde\gamma^{(t)}$ for $i>t$.
\end {defin}

\begin {lem}\label{lemma-lr-sum}
Suppose the exact sequences $E$, $\widetilde E$ have LR-tableaux $\Gamma$, $\widetilde\Gamma$,
respectively.  Then the LR-tableau of the direct sum $E\oplus \widetilde E$
is $\Gamma\cup\widetilde\Gamma$.
\end {lem}

  \noindent {\bf Proof.}
Suppose $E$, $\widetilde E$ are given by the embeddings $A\subset B$, $\widetilde A\subset \widetilde B$.
The $j$-th partition in the LR-tableau for $E\oplus\widetilde E$ is the Jordan type
for $B/T^jA \oplus \widetilde B/T^j\widetilde A$, which is $\gamma^{(j)}\cup\widetilde\gamma^{(j)}$.
  \epv

Thus, the LR-tableau of a~direct sum is obtained by merging the rows of the LR-tableaux
of the summands, starting at the top, and by sorting the entries in each row.

We present a~formula for the number $\mu_{\ell,r}$ of 
boxes $\singllebox\ell$ in the $r$-th row in the LR-tableau 
$\Gamma=[\gamma^{(0)},\ldots,\gamma^{(t)}]$ of an embedding $(A\subset B)$.
We refer to \cite[Theorem~1]{s-entries} for a~module-theoretic and homological
interpretation of this number.

\smallskip
Denote by $\gamma_{\leq r}=(\gamma'_1,\ldots,\gamma'_r)'$ the partition 
which consists of the first $r$ rows of $\gamma$.
Thus, if a~$k[T]$-module $C$ has type $\gamma$, then $C/T^rC$ has type 
$\gamma_{\leq r}$.
In particular, the first $r$ rows of the partitions $\gamma^{(\ell)}$ are given as follows.
\begin{equation} 
  \gamma^{(\ell)}_{\leq r} = \type\frac B{T^\ell A+T^r B}
\label{equation-defining-type}\end{equation}

\smallskip
As an immediate consequence, the number of boxes $\singllebox\ell$ in the first $r$ rows of $\Gamma$
is given by
$$| \gamma^{(\ell)}_{\leq r} \setminus \gamma^{(\ell-1)}_{\leq r} |
    =  \dim \frac{T^{\ell-1}A+T^r B}{T^\ell A+T^r B},$$
and the formula for $\mu_{\ell,r}$ is as follows.
\begin{eqnarray}
  \mu_{\ell,r}(A\subset B)
              & = & |\gamma^{(\ell)}_{\leq r}\setminus\gamma^{(\ell-1)}_{\leq r}| -
                    |\gamma^{(\ell)}_{\leq r-1}\setminus \gamma^{(\ell-1)}_{\leq r-1}| \\
              & = & \dim \frac{T^{\ell-1}A+T^rB}{T^\ell A+T^rB}
                    - \dim\frac{T^{\ell-1}A+T^{r-1}B}{T^\ell A+T^{r-1}B} \nonumber
\label{formula-mu}\end{eqnarray}

In the remainder of this section we study two types of examples.

\subsection{Example 1: Pickets}

\begin {defin} 
A short exact sequence $E: 0\to A\to B\to C\to 0$ is a~{\it picket}
if $B$ is indecomposable as a~$k[T]$-module (so the partition $\beta$
has only one part).  A picket $E$ is {\it empty} if $A=0$.  
\end {defin}

\begin {rem}
Recall that the invariant subspaces of a~linear operator with only one Jordan
block are determined uniquely by their dimension.
As a~consequence, a~picket $E$ as above is determined uniquely, up to isomorphy, by 
the dimensions $n=\dim B$ and $m=\dim A$.  We write
$$P^n_m\quad:=\quad (\;0\to (T^{n-m})\subset k[T]/(T^n)\to k[T]/(T^{n-m})\to 0\;).$$
\end {rem}

We picture pickets as follows. 
In the diagram, the column represents the Jordan block of $B$ and the dot
in the $(n-m+1)$-st box the submodule generator $T^{n-m}$ in $B$.
$$
\setlength\unitlength{1mm}
\raisebox{7mm}{$P_2^5:$} \quad
\begin{picture}(9,15)
  \multiput(0,0)(0,3)5{\smbox}
  \put(1.5,4.5)\sbullet
\end{picture}
\qquad\qquad
\raisebox{7mm}{$\Gamma:$} \quad
\begin{picture}(9,15)
  \multiput(0,0)(0,3)5{\smbox}
  \put(0,0){\numbox2}
  \put(0,3){\numbox1}
\end{picture}
$$

To determine the LR-tableau $\Gamma=[\gamma^{(0)},\ldots,\gamma^{(t)}]$ of a~picket, note that $t=m$, 
$\gamma^{(0)}=\type B/A=(n-m)$, $\gamma^{(1)}=\type B/TA=(n-m+1)$, $\ldots$, $\gamma^{(m)}=\type B=(n)$.

\subsection{Example 2: Poles}

\begin {defin}
A short exact sequence $E: 0\to A\to B\to C\to 0$ is a~{\it pole} if $A$ is indecomposable
as a~$k[T]$-module and $E$ is indecomposable as a~short exact sequence. 
\end {defin}

Poles have been classified, up to isomorphy, by Kaplansky \cite[Theorem~24]{kaplansky}.

\begin {thm}\label{theorem-Kaplansky}
  A pole with submodule generator $a$
  is determined uniquely, up to isomorphy, by the radical layers of the elements
  $T^ia$.
\qed 
\end {thm}

For a nonempty, strictly increasing sequence $S=(x_0,\ldots,x_{L-1})$
of nonnegative integers we construct the pole $P(S)$ for which the submodule 
generator $a$ satisfies that each $T^ia$ occurs in the $x_i$-st power of the radical.

\smallskip
Partition the sequence into intervals of subsequent numbers,
$$S=(y_1,y_1+1,\ldots,y_1+\ell_1-1,y_2,\ldots,y_2+\ell_2-1,\ldots,y_u,\ldots,y_u+\ell_u-1),$$
so $y_{i+1}>y_i+\ell_i$ for $1\leq i<u$.  Let $\beta$ be the partition
$\beta=(y_u+\ell_u,y_{u-1}+\ell_{u-1},\ldots,y_1+\ell_1)$, and put $B=N_\beta$
and $a=(T^{y_u-\ell_{u-1}-\cdots-\ell_1},\ldots,T^{y_2-\ell_1},T^{y_1})\in B$.
Then $A=(a)$ is an indecomposable $k[T]$-module and 
$P(S):0\to A\subset B\to B/A\to 0$ is an indecomposable short exact sequence such
that for each $i$, $0\leq i<L$, the element $T^ia$ is in the 
$x_i$-th radical of $B$.

\smallskip
The LR-tableau for $P(S)=[\gamma^{(0)},\ldots,\gamma^{(L)}]$ is easily computed 
as $\gamma^{(i)}\setminus\gamma^{(i-1)}$ consists of a single box $\singllebox i$
in row $x_{i-1}+1$.

\smallskip
For examples, note that each picket $P_\ell^m$ with $\ell>0$ is a pole,
more precisely, $P_\ell^m=P(m-\ell,m-\ell+1,\ldots,m-1)$.
We picture here the poles $P(0,2,3)$ and $P(0,1,3)$
and their LR-tableaux as they will occur in an example below. 
For the first pole, $\beta=(4,1)$, $a=(T,1)$; for the second $\beta=(4,2)$, $a=(T,1)$.
$$
\qquad\qquad
\begin{picture}(6,12)(0,0)
      \put(-20, 5){$P(0,2,3):$}
      \multiput(0,0)(0,3)4{\smbox}
      \multiput(3,6)(0,3)1{\smbox}
      \put(1.5,7.5)\sbullet
      \put(4.5,7.5)\sbullet
      \put(1.5,7.5){\line(1,0){3}}
    \end{picture}
\qquad\qquad\qquad
\begin{picture}(6,12)(0,0)
      \put(-15, 5){$\Gamma_{P(0,2,3)}:$}
      \multiput(0,0)(0,3)4{\smbox}
      \multiput(3,9)(0,3)1{\smbox}
      \put(1.5,1.5){\num3}
      \put(1.5,4.5){\num2}
      \put(4.5,10.5){\num1}
    \end{picture}
\qquad\qquad\qquad\qquad
\begin{picture}(6,12)(0,0)
      \put(-20, 5){$P(0,1,3):$}
      \multiput(0,0)(0,3)4{\smbox}
      \multiput(3,3)(0,3)2{\smbox}
      \put(1.5,7.5)\sbullet
      \put(4.5,7.5)\sbullet
      \put(1.5,7.5){\line(1,0){3}}
    \end{picture}
\qquad\qquad\qquad
\begin{picture}(6,12)(0,0)
      \put(-15, 5){$\Gamma_{P(0,1,3)}:$}
      \multiput(0,0)(0,3)4{\smbox}
      \multiput(3,6)(0,3)2{\smbox}
      \put(1.5,1.5){\num3}
      \put(4.5,7.5){\num2}
      \put(4.5,10.5){\num1}
    \end{picture}
$$

\section{The boundary relation and its properties} \label{section-closure}
In this section we present properties of the boundary relation defined in Formula \ref{eq-closure}.
\subsection{The boundary relation is anti-symmetric}
\label{section-closure2box}

We show that the boundary relation for LR-tableaux 
is anti-symmetric
 by verifying that it
implies the dominance order.  
This is Part (a) in Theorem \ref{theorem-bound-dom}.

\begin {lem}
\label{lem-closed-subset}
Suppose $A$, $B$ are vector spaces and 
$\mathcal M\subseteq\Hom_k(A,B)$ is a~set of monomorphisms.
For subspaces $U\subseteq A$, $V\subseteq B$ and a~natural number $n$,
the set $$ \{f\in\mathcal{M}\, :\;\; \dim(f(U)\cap V)\geq n\}$$
is closed in $\mathcal{M}$.
\end {lem}

  \noindent {\bf Proof.}
Recall that for a~natural number $m$, the condition $\rank(f)> m$ defines an open subset
in $\Hom_k(A,B)$ since it is given by the non-vanishing of a~minor 
in the matrix representing $f$. 
By restricting that matrix to a~basis for $U$ and a~basis for the
complement of $V$, we see that the condition $\dim\frac{f(U)+V}V> m$
also defines an open subset in $\Hom_k(A,B)$. 
Let now $m=\dim U-n$.  From the isomorphism $\frac{f(U)+V}V\cong \frac{f(U)}{f(U)\cap V}$
we obtain that the subset defined by $\dim\frac{f(U)}{f(U)\cap V}>m$ is open,
in particular it is open when restricted to $\mathcal M$. 
Since on $\mathcal M$, all spaces $f(U)$ have the same dimension ($f$ is a~monomorphism),
the condition is equivalent to 
$$\dim f(U)\cap V < \dim f(U)-m = n.$$
The complementary condition $\dim f(U)\cap V\geq n$ defines a~closed subset of $\mathcal M$.
  \epv

\begin {prop}
\label{proposition-critereon4closed}
For all natural numbers $i$, $\ell$, $n$, the subset 
$$\bigcup\big\{ \mathbb V_\Gamma: \Gamma \;\text{satisfies}\;
   (\gamma^{(i)})'_1+\cdots+(\gamma^{(i)})'_\ell\geq n\big\}$$
in $\mathbb V_{\alpha,\gamma}^\beta(k)$ is closed.
\end {prop}

  \noindent {\bf Proof.}
Suppose $f:A\to B$ is an embedding in $\mathbb V_\Gamma$.  Recall that the partitions
in $\Gamma$ are given by $B/f(T^iA)=N_{\gamma^{(i)}}$.

Also recall that $\dim\Hom_{k[T]}(N_{(\ell)},N_{(m)})=\min\{\ell,m\}=
\dim\frac{N_{(\ell)}}{T^mN_{(\ell)}}$.
Thus:
\begin{eqnarray*}
(\gamma^{(i)})'_1+\cdots+(\gamma^{(i)})'_\ell & = & \sum_j \min\{\gamma^{(i)}_j,\ell\}\\ 
                                     & = & \dim\Hom_{k[T]}(B/f(T^iA),N_{(\ell)})\\
                                     & = & \dim \frac{B/f(T^iA)}{T^\ell (B/f(T^iA))}\\
                                     & = & \dim \frac{B/f(T^iA)}{(T^\ell B+f(T^iA))/f(T^iA)}\\
\end{eqnarray*}
Using the isomorphism $\frac{T^\ell B+f(T^iA)}{f(T^iA)}\cong\frac{T^\ell B}{T^\ell B\cap f(T^iA)}$
we obtain
$$(\gamma^{(i)})'_1+\cdots+(\gamma^{(i)})'_\ell=\dim B-\dim f(T^iA)-\dim T^\ell B+\dim T^\ell B\cap f(T^iA).$$
Since $\dim B-\dim f(T^iA)-\dim T^\ell B=c$ is constant on $\mathbb V_{\alpha,\gamma}^\beta$,
Lemma \ref{lem-closed-subset} implies that the set
$$\bigcup\big\{ \mathbb V_\Gamma: (\gamma^{(i)})'_1+\cdots+(\gamma^{(i)})'_\ell\geq n\big\} = 
   \big\{f\in\mathbb V_{\alpha,\gamma}^\beta:\dim T^\ell B\cap f(T^iA)\geq n-c\big\}$$
is a~closed subset of $\mathbb V_{\alpha,\gamma}^\beta$.
  \epv

We can now show that the boundary relation implies the dominance order.\medskip

  \noindent {\bf Proof [of Part (a) of Theorem~\ref{theorem-bound-dom}%
]}
We assume that $\Gamma{\not\leq}_{\sf dom}\widetilde{\Gamma}$ and show that 
$\mathbb V_{\widetilde{\Gamma}}\cap\overline{\mathbb V}_\Gamma=\emptyset$.
By assumption, there exist $i$, $\ell$ such that 
$$n=(\gamma^{(i)})'_1+\cdots+(\gamma^{(i)})'_\ell>(\widetilde{\gamma}^{(i)})'_1+\cdots+(\widetilde{\gamma}^{(i)})'_\ell$$
holds.  By the proposition, 
$\mathbb U=\bigcup\big\{\mathbb V_{\widehat{\Gamma}}:(\hat{\gamma}^{(i)})'_1+\cdots+(\hat{\gamma}^{(i)})'_\ell\geq n\big\}$
is a~closed subset of $\mathbb V_{\alpha,\gamma}^\beta$ such that 
$$\mathbb V_\Gamma\subseteq \mathbb U\quad\text{and}\quad\mathbb U\cap\mathbb V_{\widetilde{\Gamma}}=\emptyset.$$
Thus, $\mathbb V_{\widetilde{\Gamma}}\cap\overline{\mathbb V}_\Gamma=\emptyset$.
  \epv

As a~consequence we obtain:

\begin {cor}
The boundary relation is reflexive and antisymmetric.\qed
\end {cor}

We conclude this section with a~result for later use.

\begin {lem}\label{lemma-invariant-subspace}
Suppose $f,g:N_\alpha\to N_\beta$ are objects in $\mathbb V_{\alpha,\gamma}^\beta$.
Let $W$ be a~subspace of $N_\beta$ which is invariant under all 
automorphisms of $N_\beta$ as a~$k[T]$-module.
If $\mathcal O_f\subset\overline{\mathcal O}_g$ then 
$$\dim \Im f\cap W\geq \dim \Im g\cap W.$$
\end {lem}

Examples of possible invariant submodules of $N_\beta$ are the powers of the
radical $T^rN_\beta$, powers of the socle $T^{-s}0$, and their intersections
$T^rN_\beta\cap T^{-s}0$.

  \noindent {\bf Proof.}
Let $h_\lambda:N_\alpha\to N_\beta$ be a~one-parameter family of objects in
$\mathbb V_{\alpha,\gamma}^\beta$ such that $h_\lambda\cong g$ for $\lambda\neq 0$
and $h_0\cong f$.  Put $n=\dim \Im g\cap W$.

Any isomorphism $h_\lambda\cong g$ ($\lambda\neq 0$) induces an isomorphism
$\Im h_\lambda\cap W\cong \Im g\cap W$ since $W$ is invariant under 
automorphisms of $N_\beta$. By Lemma~\ref{lem-closed-subset}, the set
$$\big\{h\in\mathbb V_{\alpha,\gamma}^\beta : \dim \Im h\cap W\geq n\big\}$$
is closed in $\mathbb V_{\alpha,\gamma}^\beta$, so with $h_\lambda$, $\lambda\neq 0$,
also $h_0$ is in the set. 
This shows $\dim \Im f\cap W=\dim \Im h_0\cap W\geq n$.
  \epv

\subsection{The boundary relation and the dominance relation} 
\label{section-bound-dom-not-equiv}

We have seen in Section~\ref{section-closure2box} that the boundary relation
implies the dominance relation.  Here we give an example that 
in general, the boundary relation is strictly stronger than the dominance relation.

In this and in the following section, we determine all isomorphism types of objects
which realize a given tableau that has at most 4 rows.  Such objects occur in the category
$\mathcal S(4)$ studied in \cite[(6.4)]{rs-inv} of all pairs consisting of a nilpotent linear
operator with nilpotency index at most 4 and an invariant subspace.

\begin {lem}\label{lemma-s4}
Each object in the category $\mathcal S(4)$ is a direct sum of indecomposables.
There are 20 indecomposable objects, up to isomorphy:  Four empty pickets 
$P_0^1$, $\ldots$, $P_0^4$, fifteen poles $P(S)$, where $S$ is a non-empty
subset of $\{0,1,2,3,4\}$, and a remaining object $X$ which has the property
that the invariant subspace has two Jordan blocks:
$$\begin{picture}(6,12)(0,0)
      \put(-10, 5){$X:$}
      \multiput(0,0)(0,3)4{\smbox}
      \multiput(3,3)(0,3)2{\smbox}
      \put(1.5,7.5)\sbullet
      \put(4.5,7.5)\sbullet
      \put(4.5,4.5)\sbullet
      \put(1.5,7.5){\line(1,0){3}}
    \end{picture}
\qquad\qquad\qquad
\begin{picture}(6,12)(0,0)
      \put(-10, 5){$\Gamma_X:$}
      \multiput(0,0)(0,3)4{\smbox}
      \multiput(3,6)(0,3)2{\smbox}
      \put(1.5,4.5){\num1}
      \put(1.5,1.5){\num3}
      \put(4.5,7.5){\num2}
      \put(4.5,10.5){\num1}
    \end{picture}
$$\qed
\end {lem}

Recall 
that the LR-tableau of a~direct sum is obtained by merging the rows of the LR-tableaux
of the summands, see Lemma~\ref{lemma-lr-sum}.

\begin {ex}
For $\alpha=(3,1)$, $\beta=(4,3,1)$, $\gamma=(3,1)$, there are two LR-tableaux
of shape $(\alpha,\beta,\gamma)$:
\setlength\unitlength{1mm}
$$
\begin{picture}(42,15)(0,0)
  \put(0,10){
    \begin{picture}(18,12)(0,6)
      \put(-10, 5){$\Gamma_{1}:$}
      \multiput(0,9)(3,0)2{\smbox}
      \put(6,9){\numbox{1}}
      \multiput(0,6)(3,0)1{\smbox}
      \put(3,6){\numbox{1}}
      \put(0,3){\smbox}
      \put(3,3){\numbox{2}}
      \put(0,0){\numbox{3}}
    \end{picture}
  }
  \put(30,10){
    \begin{picture}(18,12)(0,6)
      \put(-10,5){$\Gamma_{2}:$}
      \multiput(0,9)(3,0)2{\smbox}
      \put(6,9){\numbox{1}}
      \multiput(0,6)(3,0)1{\smbox}
      \put(3,6){\numbox{2}}
      \put(0,3){\smbox}
      \put(3,3){\numbox{3}}
      \put(0,0){\numbox{1}}
    \end{picture}
  }
\end{picture}
$$

We determine the possible isomorphism types of embeddings which have LR-tableaux
$\Gamma_1$ and $\Gamma_2$, respectively.  For each tableau, there is only one realization,
up to isomorphy.
$$
\qquad\qquad\qquad
\begin{picture}(6,12)(0,0)
      \put(-35, 5){$M_1=P^4_3\oplus P^3_0\oplus P^1_1:$}
      \multiput(0,0)(0,3)4{\smbox}
      \multiput(6,3)(0,3)3{\smbox}
      \multiput(12,9)(0,3)1{\smbox}
      \put(1.5,7.5)\sbullet
      \put(13.5,10.5)\sbullet
    \end{picture}
\qquad\qquad\qquad\qquad\qquad\qquad\qquad
\begin{picture}(6,12)(0,0)
      \put(-35, 5){$M_2=P^4_1\oplus P^3_3\oplus P^1_0:$}
      \multiput(0,0)(0,3)4{\smbox}
      \multiput(6,3)(0,3)3{\smbox}
      \multiput(12,9)(0,3)1{\smbox}
      \put(1.5,1.5)\sbullet
      \put(7.5,10.5)\sbullet
    \end{picture}
$$

There are no other realizations:  
Any such embedding occurs in the category $\mathcal S(4)$, so Lemma~\ref{lemma-s4}
can be used.
Considering the LR-tableau for $X$, this module cannot occur as a~summand
(since, for example, the LR-tableau for $X$ has a $\singllebox1$ in the
third row, but neither $\Gamma_1$ nor $\Gamma_2$ does).
Hence any realization is a~direct sum of poles and empty pickets.  
Note that the pole $P(0,2,3)$ cannot occur in a~decomposition for $\Gamma_1$ since 
this would require that $P^2_1$ is a~summand, which is not possible 
since there is no column of length 2 
in $\Gamma_1$.
Since each pole $P(S)$ is determined by the sequence $S$, up to isomorphy,
and since $S$ determines the entries in the LR-tableau,
there are no other choices.

As a~consequence, the varieties $\mathbb V_{\Gamma_1}$ and $\mathbb V_{\Gamma_2}$
have the same dimension, and each consists of only one orbit.  Hence
$$\mathbb V_{\Gamma_1}\cap\overline{\mathbb V}_{\Gamma_2}=\emptyset=
\mathbb V_{\Gamma_2}\cap\overline{\mathbb V}_{\Gamma_1}.$$
Thus, $\Gamma_1$ and $\Gamma_2$ are not in boundary relation, but clearly
$\Gamma_1>_{\sf dom}\Gamma_2$.

\end {ex}

\subsection{The boundary relation may not be transitive}

In general, the boundary relation  
given by 
$$\mathbb V_{\widetilde{\Gamma}}\cap \overline{\mathbb V}_\Gamma
              \neq\emptyset$$
is not transitive.  In this section, we provide an example.

\begin {ex}
\label{ex-not-transitive}
Let $\alpha=(3,1)$, $\beta=(4,3,2,1)$, $\gamma=(3,2,1)$.
There are three LR-tableaux:

\setlength\unitlength{1mm}
$$
\begin{picture}(72,15)(0,-2)
  \put(0,10){
    \begin{picture}(18,12)(0,6)
      \put(-10, 5){$\Gamma_{1}:$}
      \multiput(0,9)(3,0)3{\smbox}
      \put(9,9){\numbox{1}}
      \multiput(0,6)(3,0)2{\smbox}
      \put(6,6){\numbox{1}}
      \put(0,3){\smbox}
      \put(3,3){\numbox{2}}
      \put(0,0){\numbox{3}}
    \end{picture}
  }
  \put(30,10){
    \begin{picture}(18,12)(0,6)
      \put(-10,5){$\Gamma_{2}:$}
      \multiput(0,9)(3,0)3{\smbox}
      \put(9,9){\numbox{1}}
      \multiput(0,6)(3,0)2{\smbox}
      \put(6,6){\numbox{2}}
      \put(0,3){\smbox}
      \put(3,3){\numbox{1}}
      \put(0,0){\numbox{3}}
    \end{picture}
  }
  \put(60,10){
    \begin{picture}(18,12)(0,6)
      \put(-10,5){$\Gamma_{3}:$}
      \multiput(0,9)(3,0)3{\smbox}
      \put(9,9){\numbox{1}}
      \multiput(0,6)(3,0)2{\smbox}
      \put(6,6){\numbox{2}}
      \put(0,3){\smbox}
      \put(3,3){\numbox{3}}
      \put(0,0){\numbox{1}}
    \end{picture}
  }
\end{picture}
$$

Distributed over those three tableaux are five pairwise nonisomorphic embeddings
which can be determined using Lemma~\ref{lemma-s4}.

\begin{eqnarray*}
  M_{1}&=&P_3^4\oplus P_0^3\oplus P_0^2\oplus P_1^1, \\
  M_{12}&=&P(0,2,3)\oplus P_0^3\oplus P_1^2, \\
  M_2&=&X\oplus P_0^3\oplus P_0^1,\\
  M_{23}&=&P(0,1,3)\oplus P_1^3\oplus P_0^1, \\
  M_3&=&P_1^4\oplus P_3^3\oplus P_0^2\oplus P_0^1
\end{eqnarray*}
The notation is such that 
$M_i$ or $M_{ix}$ has LR-tableau $\Gamma_i$.

We show that the containment relation of orbit closures is as follows.\smallskip

$$
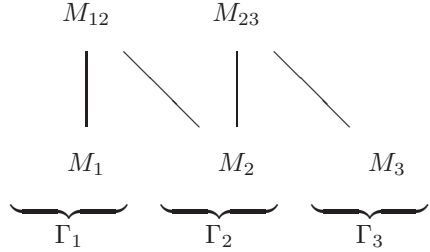
\begin{picture}(40,30)(0,0)
\put(0,10){\makebox(0,0){$M_1$}}
\put(0,30){\makebox(0,0){$M_{12}$}}
\put(20,10){\makebox(0,0){$M_2$}}
\put(20,30){\makebox(0,0){$M_{23}$}}
\put(40,10){\makebox(0,0){$M_3$}}
\put(5,25){\line(1,-1){10}}
\put(25,25){\line(1,-1){10}}
\put(0,15){\line(0,1){10}}
\put(20,15){\line(0,1){10}}
\put(-10,5){$\underbrace{\phantom{mmmmm}}_{\D\Gamma_1}$}
\put(10,5){$\underbrace{\phantom{mmmmm}}_{\D\Gamma_2}$}
\put(30,5){$\underbrace{\phantom{mmmmm}}_{\D\Gamma_3}$}
\end{picture}
$$

The short exact sequence 
$$0\longrightarrow P_1^2 \longrightarrow M_2 \longrightarrow P(0,2,3)\oplus P_0^3
   \longrightarrow 0$$
shows that $\mathcal O(M_{12})\subset \overline{\mathcal O}(M_2)$ 
(since the ext-order implies the degeneration order, see Section~\ref{section-deg}). Hence
$\mathbb V_{\Gamma_1}\cap \overline{\mathbb V}_{\Gamma_2}\neq \emptyset$
and $\Gamma_1>_{\sf boundary}\Gamma_2$.

Similarly, the short exact sequence 
$$0 \longrightarrow P_1^3 \longrightarrow M_3 \longrightarrow P(0,1,3)\oplus P_0^1
    \longrightarrow 0$$
shows that $\mathcal O(M_{23})\subset \overline{\mathcal O}(M_3)$, hence
$\mathbb V_{\Gamma_2}\cap \overline{\mathbb V}_{\Gamma_3}\neq \emptyset$  and
$\Gamma_2>_{\sf boundary}\Gamma_3$. 

However, $\mathbb V_{\Gamma_1}\cap \overline{\mathbb V}_{\Gamma_3}=\emptyset$.
The only possible orbit in the intersection is $\mathcal O(M_{12})$, since there
are only two orbits in $\mathbb V_{\Gamma_1}$, and since 
the other orbit $\mathcal O(M_1)$ has the  same dimension as 
$\mathbb V_{\Gamma_3}=\mathcal O(M_3)$. 

Note that the module $M_{12}=(U\subset V)$ has the property that 
$\dim U\cap T^2V\cap\soc V=1$,
while for the module $M_3$, the corresponding dimension is $2$.
It follows from Lemma~\ref{lemma-invariant-subspace} with $W=T^2V\cap\soc V$ 
that $\mathcal O(M_{12})\not\subseteq\overline{\mathcal O}(M_3)$.

This finishes the example which illustrates that in general, the condition for
LR-tableaux that $\mathbb V_{\widetilde{\Gamma}}\cap\overline{\mathbb V}_\Gamma\neq \emptyset$
may not define a~partial order. 
\qed
\end {ex}

\section{The algebraic orders for LR-tableaux}
\label{section-deg}

For modules of a~fixed dimension over a~finite dimensional algebra
the three partial orders 
$$\leq_{\sf ext},\quad \leq_{\sf deg},\quad \leq_{\sf hom}$$
have been studied extensively, see for example \cite{bongartz,bongartz1,riedtmann,kos03,zwara}.
In particular, the partial orders are available for invariant subspaces 
in $\mathbb V_{\alpha,\gamma}^\beta$, see \cite[Section~3.2]{ks-hall}.
For the convenience of the reader we recall these definitions. Let $f,g\in\mathbb{V}_{\alpha,\gamma}^\beta$.
 \begin{itemize}
 \item The relation $f \leq_{\sf ext} g$  holds if there exist 
embeddings  $h_i$, $u_i$, $v_i$ of linear operators and short exact
sequences $0\to u_i\to h_i\to v_i\to 0$ of embeddings  
such that $f\cong h_1$, $u_i\oplus v_i\cong h_{i+1}$ for $1\leq i\leq s$, 
and $g\cong h_{s+1}$, for some natural
number $s$.

\item The relation $f \leq_{\sf deg} g$ 
holds if $\mathcal{O}_g \subseteq \overline{\mathcal{O}_f}$ in 
$\mathbb V_{\alpha,\gamma}^\beta(k)$.

\item The relation $f\leq_{\sf hom} g$  holds
if $$[f,h]\leq [g,h]$$ 
for any embedding $h$,
where $[f,h]$ denotes the dimension of the linear space $\Hom(f,h)$
of all homomorphisms of embeddings. 
\end{itemize}

\smallskip
They induce three reflexive and anti-symmetric relations on the set $\mathcal T_{\alpha,\gamma}^\beta$.

\begin {defin}
  Suppose $\Gamma$, $\widetilde{\Gamma}$ are two LR-tableaux of shape $(\alpha,\beta,\gamma)$.
      We write $\Gamma \leq_{\sf ext}\widetilde{\Gamma}$
      ($\Gamma \leq_{\sf deg}\widetilde{\Gamma}$; $\Gamma \leq_{\sf hom}\widetilde{\Gamma}$)
      if there is a~sequence 
      $$\Gamma=\Gamma^{(0)}, \Gamma^{(1)},\ldots,\Gamma^{(s)}=\widetilde{\Gamma}$$
      such that for each $1\leq i\leq s$ there are $f\in \mathbb V_{\Gamma^{(i-1)}}$,
      $g\in \mathbb V_{\Gamma^{(i)}}$ with $f\leq_{\sf ext} g$
      ($f\leq_{\sf deg} g$; $f\leq_{\sf hom} g$).
\end {defin}

It follows from the corresponding properties for modules that:
\begin{itemize}
\item $\Gamma\leq_{\sf ext}\widetilde{\Gamma}$ implies $\Gamma\leq_{\sf deg}\widetilde{\Gamma}$ and 
\item $\Gamma\leq_{\sf deg}\widetilde{\Gamma}$ implies $\Gamma\leq_{\sf hom}\widetilde{\Gamma}$.
\end{itemize}
Also, it is clear from the definitions that
\begin{itemize}
\item  $\Gamma\leq_{\sf deg}\widetilde{\Gamma}$ implies $\Gamma\leq_{\sf boundary}\widetilde{\Gamma}$.
\end{itemize}

We have seen in Section~\ref{section-closure2box} that the boundary relation
implies the dominance order $\leq_{\sf dom}$.  
In the following section we show that also the hom-relation implies the dominance order.
As a consequence, each of the three relations
$\leq_{\sf ext}, \;\leq_{\sf deg}, \;\leq_{\sf hom}$ is anti-symmetric,
hence a partial ordering.

\subsection{The Hom-relation implies the dominance order}

We start with an abstract result.

Denote by $\mathcal N$ the category $\mod k[T]_{(T)}$ of all nilpotent linear operators,
and by $\mathcal S=\mathcal S(k[T]_{(T)})$ the category of all invariant subspaces.
For each $i\in\mathbb N$, there is a~pair of functors
\begin{eqnarray*}
R_i: & \mathcal S\to \mathcal N, & (A\subset B)\mapsto \frac B{T^iA}\\
L_i: & \mathcal N\to \mathcal S, & X \mapsto (\soc^iX\subset X).
\end{eqnarray*}

\begin {lem}
\label{lemma-adjoint}
For each $i\in\mathbb N$, the functors $R_i$, $L_i$ form an adjoint pair.
\end {lem}

  \noindent {\bf Proof.}
Given an operator $X\in\mathcal N$ and an invariant subspace $(A\subset B)\in\mathcal S$,
we need to show that there is a~natural isomorphism
$$\Hom_{\mathcal S}((A\subset B),L_i(X))\;\cong\;\Hom_{\mathcal N}(R_i(A\subset B),X).$$
A morphism in $\mathcal S$ is given by a~commutative diagram:
$$\begin{CD}
 A @>f|_A>> \soc^iX\\
 @VVV @VVV \\
 B @>>f> X
\end{CD}$$
It gives rise to the commutative diagram:
$$\begin{CD}
 \rad^iA @>>> 0\\
@VVV @VVV \\
 B @>>f> X
\end{CD}$$
Hence we obtain a~morphism in $\mathcal N$:
$$\bar f:\frac B{\rad^iA}\; \longarr{}{40}\; X.$$
Conversely, the morphism in $\mathcal N$ gives rise to a~commutative diagram
and hence to a~morphism in $\mathcal S$. 
Clearly, the two constructions are inverse to each other.
  \epv

We recognize that the objects of the form  $P_i^\ell=L_i(N_{(\ell)})$ are pickets.

\begin {prop}\label{prop-hom2dom}
Suppose the objects $(A\subset B)$ and $(\widetilde{A}\subset \widetilde{B})$ have LR-tableaux
$\Gamma$ and $\widetilde{\Gamma}$, respectively.  The following assertions are equivalent:
\begin{enumerate}
\item $\Gamma\leq_{\sf dom}\widetilde{\Gamma}$
\item For each picket $P_i^\ell$ the inequality holds:
$$\dim\Hom_{\mathcal S}((A\subset B),P_i^\ell)\leq\dim\Hom_{\mathcal S}((\widetilde{A}\subset\widetilde{B}),P_i^\ell)$$
\end{enumerate}
\end {prop}

  \noindent {\bf Proof.}
By the definition given in Section~\ref{section-lr-orders},
the condition 
$\Gamma\leq_{\sf dom}\widetilde{\Gamma}$ is equivalent to 
$$(\gamma^{(i)})'_1+\cdots+(\gamma^{(i)})'_\ell\leq(\widetilde{\gamma}^{(i)})'_1+\cdots+(\widetilde{\gamma}^{(i)})'_\ell\qquad
  \text{for each $i$ and $\ell$}.$$
Let $i$ and $\ell$ be natural numbers.  
We obtain from Lemma~\ref{lemma-adjoint} and from the equality in the proof of
Proposition~\ref{proposition-critereon4closed} that
$$(\gamma^{(i)})'_1+\cdots+(\gamma^{(i)})'_\ell=\dim\Hom_{\mathcal N}(B/T^iA,N_{(\ell)})
                                      =\dim\Hom_{\mathcal S}((A\subset B),P_i^\ell)$$
The claim follows from this and from the corresponding equality for $(\widetilde{A}\subset \widetilde{B})$.
  \epv

It follows that the  
restriction $\leq_{\sf hom-picket}$ of the hom order to pickets and the dominance relation
are equivalent.  Hence, the hom-relation implies the dominance order.
Without imposing any conditions on the triple $(\alpha,\beta,\gamma)$, we have established
the following implications:
\vspace{-1cm}


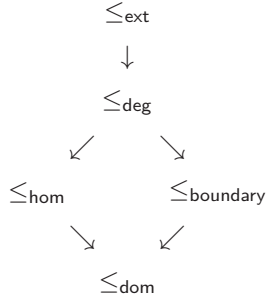
$$
\setlength\unitlength{.6mm}
\begin{picture}(40,80)
\put(20,0){\makebox(0,0){$\leq_{\sf dom}$}}
\put(0,20){\makebox(0,0){$\leq_{\sf hom}$}}
\put(40,20){\makebox(0,0){$\leq_{\sf boundary}$}}
\put(20,40){\makebox(0,0){$\leq_{\sf deg}$}}
\put(20,60){\makebox(0,0){$\leq_{\sf ext}$}}
\put(20,50){\makebox(0,0){$\downarrow$}}
\put(10,30){\makebox(0,0){$\swarrow$}}
\put(30,30){\makebox(0,0){$\searrow$}}
\put(10,10){\makebox(0,0){$\searrow$}}
\put(30,10){\makebox(0,0){$\swarrow$}}
\end{picture}
$$

\subsection{The box-order implies the ext-order (for horizontal strips)}
\label{section-box-ext}

Assume that $\Gamma,\widetilde\Gamma$ are LR-tableaux of the shape $(\alpha,\beta,\gamma)$
such that $\beta\setminus\gamma$ is a~horizontal strip. In \cite{ks-ext} we prove the following. 
Suppose that the Littlewood-Richardson tableau $\widetilde\Gamma$
is obtained from $\Gamma$ by an increasing box move.
By \cite[Proposition 6.1]{ks-ext}, there exist
embeddings $M,\widetilde M$ that realize tableaux $\Gamma,\widetilde \Gamma$,
respectively, and such that $M=R\oplus R'\oplus N$, $\widetilde M=\widetilde R\oplus \widetilde R'\oplus N$
for some embeddings $R,R'\widetilde R,\widetilde R',N$. In \cite[Sections 7,8]{ks-ext}
it is shown that there exists a~short exact sequence
$$
0\longrightarrow \widetilde R\longrightarrow Q \longrightarrow \widetilde R'\longrightarrow 0
$$
for some embedding $Q$ with the same LR-tableau as $R\oplus R'$. Therefore $Q\leq_{\sf ext} \widetilde R\oplus \widetilde R'$
and $\Gamma\leq_{\sf ext}\widetilde \Gamma$.

With the assumption that $\beta\setminus\gamma$ is~a horizontal strip we have the following implications:

$$
\setlength\unitlength{.6mm}
\begin{picture}(40,80)
\put(20,0){\makebox(0,0){$\leq_{\sf dom}$}}
\put(0,20){\makebox(0,0){$\leq_{\sf hom}$}}
\put(40,20){\makebox(0,0){$\leq_{\sf boundary}$}}
\put(20,40){\makebox(0,0){$\leq_{\sf deg}$}}
\put(20,60){\makebox(0,0){$\leq_{\sf ext}$}}
\put(20,80){\makebox(0,0){$\leq_{\sf box}$}}
\put(20,70){\makebox(0,0){$\downarrow$}}
\put(20,50){\makebox(0,0){$\downarrow$}}
\put(10,30){\makebox(0,0){$\swarrow$}}
\put(30,30){\makebox(0,0){$\searrow$}}
\put(10,10){\makebox(0,0){$\searrow$}}
\put(30,10){\makebox(0,0){$\swarrow$}}
\end{picture}
$$

\begin {rem}
 \begin{enumerate}
  \item In \cite{ks-ext} the implication $\leq_{\sf box}\;\Longrightarrow\; \leq_{\sf ext}$ is proved in a~more general case.
  \item The Example 2 in Section \ref{section-lr-orders} shows that these orders are not equivalent in general 
  (even for horizontal strips).
  \item Results of \cite{kst} prove the equivalence of all these orders in the case $\beta\setminus\gamma$
  is a~horizontal and vertical strip (compare Theorem \ref{theorem-horizontal-vertical}).
 \end{enumerate}
\end {rem}

\subsection{The ext- and deg-relations are not equivalent}
\label{section-ext-deg}

It is well-known that for modules, the ext-relation $\leq_{\sf ext}$
implies the deg-relation $\leq_{\sf deg}$. 
In general for modules, the converse is not the case.
Here we give an example for embeddings of linear operators.

\begin {ex}
For $\alpha=(4,2)$, $\beta=(6,4,2)$, $\gamma=(4,2)$, there are three
LR-tableaux:
$$
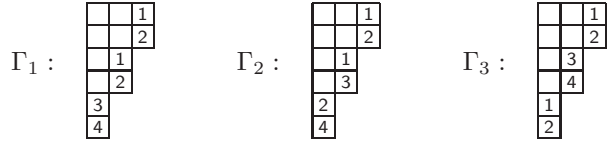
\begin{picture}(79,18)(0,0)
  \put(10,0){
    \begin{picture}(9,18)(0,0)
      \put(-10, 9){$\Gamma_{1}:$}
      \multiput(0,15)(0,-3)4{\smbox}
      \put(0,3){\numbox{3}}
      \put(0,0){\numbox{4}}
      \multiput(3,15)(0,-3)2{\smbox}
      \put(3,9){\numbox{1}}
      \put(3,6){\numbox{2}}
      \put(6,15){\numbox{1}}
      \put(6,12){\numbox{2}}
    \end{picture}
  }
  \put(40,0){
    \begin{picture}(9,18)(0,0)
      \put(-10, 9){$\Gamma_{2}:$}
      \multiput(0,15)(0,-3)4{\smbox}
      \put(0,3){\numbox{2}}
      \put(0,0){\numbox{4}}
      \multiput(3,15)(0,-3)2{\smbox}
      \put(3,9){\numbox{1}}
      \put(3,6){\numbox{3}}
      \put(6,15){\numbox{1}}
      \put(6,12){\numbox{2}}
    \end{picture}
  }
  \put(70,0){
    \begin{picture}(9,18)(0,0)
      \put(-10, 9){$\Gamma_{3}:$}
      \multiput(0,15)(0,-3)4{\smbox}
      \put(0,3){\numbox{1}}
      \put(0,0){\numbox{2}}
      \multiput(3,15)(0,-3)2{\smbox}
      \put(3,9){\numbox{3}}
      \put(3,6){\numbox{4}}
      \put(6,15){\numbox{1}}
      \put(6,12){\numbox{2}}
    \end{picture}
  }
\end{picture}
$$
We show that the partial orders given by $\leq_{\sf ext}$ and $\leq_{\sf deg}$
are as follows:

$$
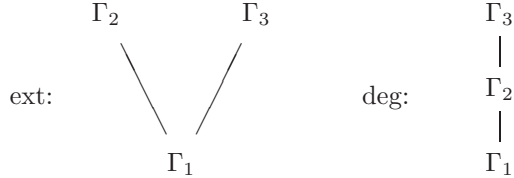
\raisebox{8mm}{\rm ext:}\qquad
\begin{picture}(20,20)
\put(10,0){\makebox(0,0){$\Gamma_1$}}
\put(0,20){\makebox(0,0){$\Gamma_2$}}
\put(20,20){\makebox(0,0){$\Gamma_3$}}
\put(8,4){\line(-1,2)6}
\put(12,4){\line(1,2)6}
\end{picture}
\qquad\qquad
\raisebox{8mm}{\rm deg:}\qquad
\begin{picture}(10,20)
\put(5,0){\makebox(0,0){$\Gamma_1$}}
\put(5,10){\makebox(0,0){$\Gamma_2$}}
\put(5,20){\makebox(0,0){$\Gamma_3$}}
\put(5,3){\line(0,1)4}
\put(5,13){\line(0,1)4}
\end{picture}
$$
(In each case, $\Gamma_1$ is the largest element in the poset.)

First we describe the embeddings which realize the tableaux.
From \cite{rs-inv} we know that there is a~one-parameter family
of indecomposable embeddings $M_2(\lambda)$ occurring on the mouths of the homogeneous
tubes with tubular index 0; they all have type $\Gamma_2$.
There are two additional indecomposables, they occur in the 
tube of circumference 2 at index 0; the modules are dual to each other
and have type $\Gamma_1$ and $\Gamma_2$, respectively.
We sketch the modules, using the conventions as in \cite{rs-inv}.
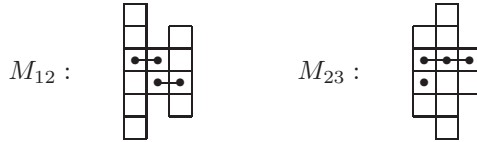
$$
\raisebox{8mm}{$M_{12}:$}\qquad
\begin{picture}(9,18)(0,0)
      \multiput(0,0)(0,3)6{\smbox}
      \multiput(3,6)(0,3)2{\smbox}
      \multiput(6,3)(0,3)4{\smbox}
      \put(1.5,10.5)\sbullet
      \put(4.5,10.5)\sbullet
      \put(4.5,7.5)\sbullet
      \put(7.5,7.5)\sbullet
      \put(1.5,10.5){\line(1,0){3}}
      \put(4.5,7.5){\line(1,0){3}}
    \end{picture}
\qquad\qquad
\raisebox{8mm}{$M_{23}:$}\qquad
\begin{picture}(9,18)(0,0)
      \multiput(0,3)(0,3)4{\smbox}
      \multiput(3,0)(0,3)6{\smbox}
      \multiput(6,6)(0,3)2{\smbox}
      \put(1.5,10.5)\sbullet
      \put(4.5,10.5)\sbullet
      \put(7.5,10.5)\sbullet
      \put(1.5,7.5)\sbullet
      \put(1.5,10.5){\line(1,0){6}}
    \end{picture}
$$

In addition, there are three decomposable configurations; note that 
$M_1$ is the dual of $M_3$ while $M_{123}$ is self dual.
$$
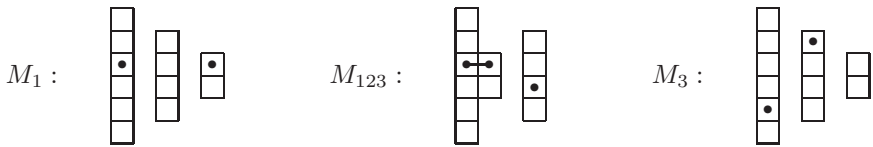
\raisebox{8mm}{$M_{1}:$}\qquad
\begin{picture}(15,18)(0,0)
      \multiput(0,0)(0,3)6{\smbox}
      \multiput(6,3)(0,3)4{\smbox}
      \multiput(12,6)(0,3)2{\smbox}
      \put(1.5,10.5)\sbullet
      \put(13.5,10.5)\sbullet
    \end{picture}
\qquad\qquad
\raisebox{8mm}{$M_{123}:$}\qquad
\begin{picture}(12,18)(0,0)
      \multiput(0,0)(0,3)6{\smbox}
      \multiput(3,6)(0,3)2{\smbox}
      \multiput(9,3)(0,3)4{\smbox}
      \put(1.5,10.5)\sbullet
      \put(4.5,10.5)\sbullet
      \put(1.5,10.5){\line(1,0){3}}
      \put(10.5,7.5)\sbullet
    \end{picture}
\qquad\qquad
\raisebox{8mm}{$M_{3}:$}\qquad
\begin{picture}(15,18)(0,0)
      \multiput(0,0)(0,3)6{\smbox}
      \multiput(6,3)(0,3)4{\smbox}
      \multiput(12,6)(0,3)2{\smbox}
      \put(1.5,4.5)\sbullet
      \put(7.5,13.5)\sbullet
    \end{picture}
$$

The modules $M_1=P_4^6\oplus P_0^4\oplus P_2^2$ and $M_{123}=P(0,1,4,5)\oplus P_2^4$
have type $\Gamma_1$, and $M_3=P_2^6\oplus P_4^4\oplus P_0^2$ has type $\Gamma_3$.

We claim that there are no further isomorphism types of objects in 
$\mathbb V_{\alpha,\gamma}^\beta$.

\smallskip
For finite fields, the Hall polynomial $g_{\alpha,\gamma}^\beta$ counts the number
of submodules of $N_\beta$ which are isomorphic to $N_\alpha$ and have factor $N_\gamma$.
For each of the isomorphism types of embeddings (that is, $M_1$, $M_{12}$, $M_{123}$,
$M_2(\lambda)$ ($\lambda\neq 0,1$), $M_{23}$, $M_3$), we can count the corresponding
numbers of submodules of $N_\beta$.  It is straightforward to verify that the sum,
taken over the isomorphism types, is exactly $g_{\alpha,\gamma}^\beta$. 

\smallskip
For algebraically closed fields, the embeddings $M_1$, $M_{123}$, $M_3$ are sums of
exceptional objects in the covering category $\mathcal S(\widetilde 6)$ studied in 
\cite{rs-inv}, the others are indecomposable non-exceptional objects.
The $M_2(\lambda)$ occur in the homogeneous tubes, $M_{12}$ and $M_{23}$ 
in the tube of circumference 2 in the tubular family of index 0; 
the remaining tubes of index 0 are pictured in \cite[(2.3)]{rs-inv}, they contain
no non-exceptional objects in $\mathbb V_{\alpha,\gamma}^\beta$.
All non-exceptional objects in tubes of index different from 0 have higher dimension.
It follows that each remaining object in $\mathbb V_{\alpha,\gamma}^\beta$ is a direct sum
of exceptional modules. Each exceptional object $X$ is determined uniquely by its dimension
vector in $\mathcal S(\widetilde 6)$ and can be realized over any field.  
The dimension of the homomorphism spaces $\Hom(P,X)$ where $P$ is a picket, 
and hence the LR-tableau for $X$ (\cite{s-entries}) do not depend on the base field.
Hence $M_1$, $M_{123}$ and $M_3$ are the only objects in $\mathbb V_{\alpha,\gamma}^\beta$ which
have an exceptional direct summand.

We determine the ext-order and the deg-order on $\mathcal T_{\alpha,\gamma}^\beta$.

Consider the short exact sequences
$$0 \longrightarrow P_2^4\longrightarrow M_{23}\longrightarrow P(0,1,4,5)\longrightarrow 0$$
and
$$0\longrightarrow P_2^4\longrightarrow M_3\longrightarrow P(0,1,4,5)\longrightarrow 0.$$
In each, the sum of the end terms is $M_{123}$.  It follows that 
$\Gamma_1\geq_{\sf ext}\Gamma_2$ and $\Gamma_1\geq_{\sf ext}\Gamma_3$, respectively.
Note that $\Gamma_2{\not>}_{\sf ext}\Gamma_3$ since there is no decomposable module of type $\Gamma_2$.

Since the ext-relation implies the deg-relation, it remains to show that $\Gamma_2\geq_{\sf deg}\Gamma_3$.
As mentioned, the modules $M_1$ and $M_3$ are dual to each other, so their orbits have the same dimension.
As $\mathcal O_{M_3}=\mathbb V_{\Gamma_3}$, and since all varieties given by LR-tableaux are
irreducible of the same
dimension, it follows that $\mathcal O_{M_1}$ is dense in $\mathbb V_{\Gamma_1}$. 
In particular, $\mathcal O_{M_1}$ contains $\mathcal O_{M_{12}}$ in its closure.
Applying duality again, we obtain that $\mathcal O_{M_3}$ contains $\mathcal O_{M_{23}}$ 
in its closure.  Thus, $\mathcal O_{M_{23}}$ is in the closure of $\mathbb V_{\Gamma_3}$.
\qed
\end {ex}

{\bf Acknowledgements}
The authors are indebted to Hugh Thomas for discussions which led
to the proof of Theorem~\ref{theorem-dom-box}, 
and for his contribution of the example
in Section~\ref{section-bound-dom-not-equiv}.

\smallskip
This research project was started when the authors visited the Mathematische 
Forschungsinstitut Oberwolfach in a Research in Pairs project.
They would like to thank the members of the 
Forschungsinstitut for creating an outstanding
environment for their mathematical research.



\end{document}